\documentclass[12pt]{article}
\usepackage{amsmath, amsfonts, graphicx, amsthm, color, enumerate, hyperref, amssymb, float}

 \usepackage{tikz}
 \usepackage{tikz-cd}

\numberwithin{equation}{section}

\makeatletter
\newsavebox{\@brx}
\newcommand{\llangle}[1][]{\savebox{\@brx}{\(\m@th{#1\langle}\)}%
  \mathopen{\copy\@brx\kern-0.5\wd\@brx\usebox{\@brx}}}
\newcommand{\rrangle}[1][]{\savebox{\@brx}{\(\m@th{#1\rangle}\)}%
  \mathclose{\copy\@brx\kern-0.5\wd\@brx\usebox{\@brx}}}
\makeatother

\DeclareMathOperator{\Aut}{Aut}

\DeclareMathOperator{\id}{id}
\DeclareMathOperator{\ch}{char}

\DeclareMathOperator{\GL }{GL }

\DeclareMathOperator{\SL}{SL}
\DeclareMathOperator{\Sp}{Sp}

\DeclareMathOperator{\Mat}{Mat}
\DeclareMathOperator{\Orth}{O}

\DeclareMathOperator{\G}{G}

\DeclareMathOperator{\q}{q}

\newtheorem{thm}[subsection]{Theorem}
\newtheorem{prop}[subsection]{Proposition}
\newtheorem{lem}[subsection]{Lemma}
\newtheorem{cor}[subsection]{Corollary}

\newtheorem{ex}[subsection]{Example}

\newcommand{\bi}[1]{\langle #1 \rangle}

\title{Involutions of type $\G_2$ over a field of characteristic two}
\author{John Hutchens \\ \textit{Winston-Salem State University} \\ hutchensjd@wssu.edu \\  \\ Nathaniel Schwartz \\ \textit{Washington College} \\ nschwartz2@washcoll.edu}

\begin{document}

\maketitle

\section{Introduction}

In this paper we establish a classification of isomorphism classes of automorphisms of order $2$ of algebraic groups of the form $\Aut(C)$ where $C$ is an octonion algebra defined over a field of characteristic $2$ and their fixed point groups.  Groups of type $\G_2$ were first defined over fields of characteristic $2$ by Dickson in \cite{di01,di05} and later by Chevalley in \cite{ch55}.  For fields of not characteristic $2$, a similar classification has been carried out by Yokota in \cite{yo90} and Hutchens in \cite{hu14}.

The orginal study of symmetric spaces began with Cartan \cite{car26, car27} in 1926-27, followed by Gantmacher in \cite{ga39}.  Berger also classifies symmetric spaces of real Lie algebras in \cite{be57}.   Helminck defined a generalization of these real symmetric spaces and called them symmetric $k$-varieties.  He provided a characterization of these spaces in \cite{he00,he88} over fields of not characteristic $2$.  For finite fields this problem was solved by Thomas in \cite{th69}, and Aschbacher and Seitz included these results to complete their classification of such structures for groups of Lie type over fields of even order in \cite{as76}.  What we call involutions of type I correspond to Thomas' central involutions, and our involutions of type II correspond to his non-central involutions over finite fields.  Schwartz has provided descriptions for algebraic groups of classical type $A_n$ over fields of characteristic $2$ in \cite{sc16}, and Hutchens has over fields of characteristic not $2$ for algebraic groups of type $\G_2$ in \cite{hu14}.  

In \ref{class} we prove the classification of $k$-involutions depends on isomorphism classes of four dimensional subalgebras of $C$.  This is also true when the characteristic is not $2$, which Jacobson proves in \cite{ja58}.  His argument relies on the fact that there is a $(-1)$-eigenspace for each element of order $2$ in $\Aut(C)$, the existence of which we forfeit over characteristic $2$.  Also, four dimensional subalgebras of $C$ are all of quaternion type when the characteristic is not $2$.  In the characteristic $2$ case we have another type of four dimensional subalgebra that corresponds to elements of order $2$ in $\Aut(C)$.   A representative of this second type of involution is used by Callens and de Medts in \cite{cdm14} to construct a class of Moufang sets.  Recently there has been a lot interest in topics related to algebras and algebraic groups over characteristic $2$, for example \cite{no16,dol14,bc15}.

This paper is organized as follows. \ref{prereqs} contains a brief background on composition algebras, including necessary prerequisites. Here, we generally follow the conventions and notations used by Hoffmann and Laghribi in \cite{hl04}. In \ref{4dalgs} we settle on a basis for our octonion algebra $C$, and we describe two different four dimensional subalgebras: quaternion subalgebras and four-dimensional totally singular subalgebras. The latter is actually a Clifford algebra that only occurs in characteristic 2. We show that every involution in $\Aut(C)$ fixes point-wise one of these two types of subalgebras.  We call the involutions that fix a quaternion subalgebra {\em type I}, and involutions that fix a four-dimensional totally singular subalgebra {\em type II}. In \ref{typeI}, we see that there is only one conjugacy class of involutions of type I in the case that the fixed subalgebra is split, and there are no involutions when the fixed subalgebra is a division quaternion subalgebra. We show in \ref{fixpt1} that the fixed point groups of type I are of the form $\SL_2(k) \times G_+(k)$, and subsequently describe the fixed point groups of type II. 

In \ref{typeII} we describe the involutions in $\Aut(C)$ fixing a four dimensional totally singular subalgebra $B$. We show in \ref{tinduce} that these involutions are induced by elements in the set $\widehat{B}$  of elements whose norm is equal to the coefficient of the identity component. In \ref{gfixB} we show every such involution acts invariantly on $B$. In \ref{gidcor} we establish that, if $B$ is a division algebra, then its automorphism group is trivial. \ref{autCconjugate} summarizes how two involutions can be conjugate when the fixed point algebra is division.  We find that all involutions are conjugate to the map $u \mapsto u + e$; hence there is a single conjugacy class of involutions when $B$ is not a division algebra. The entire situation is summarized in \ref{classesBsummary}. A short discussion of the corresponding Galois cohomology follows.

\section{Preliminaries and composition algebras}
\label{prereqs}

If $k$ is a field of characteristic $2$, we call $k$ a \emph{perfect field} if $k^2=k$.  This is the case for all finite and algebraically closed fields of characteristic $2$. For a classification of involutions and their centralizers of $\Aut(C)$ over finite fields see \cite{as76} and \cite{th69}.  

We denote by $\Aut(C,B)$ the subgroup of $\Aut(C)$ leaving a subaglebra $B \subset C$ invariant and by $\Aut(C)^B$ the subgroup of $\Aut(C)$ leaving $B$ fixed elementwise.  The additive group of a field $k$ will be denoted by $G_+(k)$ and the multiplicative group of $k$ is denoted by $G_*(k)$.  When $\ch(k)=2$ the group $G_+(k)$ has subgroups $G_+(k^2)$ and
\[
\wp(k) = \{ x^2 + x \ | \ x \in k \}.
\]

When we consider groups of type $\G_2$ over fields of characteristic not $2$ they are induced by an element of an automorphism group of an octonion algebra over $k$ that fixes a quaternion subalgebra.  We see in \cite{hu14} that the isomorphism classes of quaternion subalgebras of an octonion algebra $C$ are in bijection with the conjugacy classes of $k$-involutions of $\Aut(C)$.  Also, the fixed point group of these $k$-involutions of $\Aut(C)$ have the form
\[ \Aut(C)^{\mathcal{I}_t} \cong \Aut(D) \ltimes \Sp(1,k), \]
where $D$ is the quaternion subalgebra of $C$ fixed by $t\in \Aut(C)$ and $\Sp(1,k)$ is the group of norm $1$ quaternions in $D$. We have a similar situation in characteristic $2$.  In \ref{mainlem1} and \ref{mainlem2} we show that involutions of $\Aut(C)$ in characteristic $2$ are induced by elements in $\Aut(C)$ that fix four dimensional subalgebras of $C$, and in \ref{fixpt1} we see that the fixed point groups take a slightly different form than in the characteristic not 2 case.

For the remainder of this paper we will only consider fields of characteristic $2$ unless otherwise stated. Let $k$ be a field of characteristic 2 and $V$ be a vector space over $k$; then $\q:V \to k$ is a \emph{quadratic form} on $V$ if
\[
\q( \lambda x ) = \lambda^2 \q(x) 
\]
for all $\lambda \in k$ and $v \in V$, such that 
\[
\q(x + y) + \q(x) + \q(y) = b(x,y)
\]
for all $x, y \in V$, where $\bi{\ ,\ }: V\times V \to k$ is a symmetric bilinear form. 

A quadratic form is \emph{non-degenerate} if the associated bilinear form is non-degenerate; that is, if $V^{\perp} = 0$. In other words, $\bi{\ ,\ }$ is non-degenerate if $\bi{x,y} = 0$ for all $x,y \in V$ implies $x = 0$.  A nonzero vector $v \in V$ is called isotropic if $x \in V^{\perp}$ and $\q(x) = 0$. 
Since $0 = \q(0) = \q(x + x + \q(x) + \q(x) = \bi{x,x}$, $\bi{\ ,\ }$ is also alternating. Note that $\bi{\ ,\ }$ is uniquely determined from $\q$, but $\q$ is not uniquely determined by $\bi{\ ,\ }$.

Suppose $V$ has dimension $2r + s$. Then there exists a basis for $V$ of the form $\{ e_i, f_i, g_j \}$ and elements $a_i, b_i, c_j \in k$ with $1 \le i \le r$, $1 \le j \le s$, such that
\[
\q(v) = \sum_{i=1}^r \left(a_ix_i^2 + x_iy_i + b_iy_i^2\right) + \sum_{j = 1}^s c_j z_j^2
\]
for any $v = \sum_{i = 1}^2 (x_ie_i + y_ij_i) + \sum_{j = 1}^s x_jg_j$. We denote this quadratic form by
\[
[a_1, b_1]\perp[a_2,b_2]\perp \cdots \perp [a_r, b_r] \perp \langle c_1, \dots, c_s\rangle,
\]
where $\langle c_1, \dots, c_s\rangle$ is called the {\it diagonal } part of $\q$. 

An {\it isometry} of two quadratic spaces $(V_1, \q_1)$ and $(V_2, \q_2)$ is a bijection $\sigma : V_1 \to V_2$ such that $\q_2(\sigma(v)) = q_1(v)$ for all $v \in V_1$. If such an isometry exists, we say that $\q_1$ and $\q_2$ are equivalent.

As noted in \cite{hl04} the following equivalences hold for all $a,b,c,d \in k$ and $x \in k^*$:
\begin{align}
\bi{a} &\cong \bi{x^2 a} \\
[a,b] &\cong [ax^2, bx^{-2}] \\
\bi{a,b} &\cong \bi{a,a+b} \cong \bi{b,a} \\
[a,b] &\cong [a, a + b + 1] \cong [b, a] \\
[a,b]\perp \bi{c} &\cong [a + c, b] \perp \bi{c}\\
[a,b]\perp [c,d] &\cong [a + c, b] \perp [c, b + d].
\end{align}

An {\em$n$-fold Pfister form} is a form of type $b = \bi{1,a_1}_b\otimes \cdots \otimes \bi{1,a_n}_b$ for some $a_i \in G_*(k)$. In this case, the quadratic form is totally singular, which implies the form is quasi-linear, and we write $\q = \llangle a_1, \dots, a_n\rrangle$. The form $\q$ is called an {\em $n$-fold quasi-Pfister form}.

A {\em hyperbolic plane} $\mathbb{H}$ is a non-singular form isometric to $[0,0]$. Any non-singular form which is an orthogonal sum of hyperbolic planes is called hyperbolic.

A {\em composition algebra} is a $k$-algebra with an identity element $e$ equipped with a nondegenerate quadratic form $\q:C \to k$ which permits composition, i.e.
\[ \q(xy) = \q(x)\q(y), \]
for all $x,y \in C$.  

It was shown by Hurwitz \cite{hu22} that such algebras exist only in $1, 2, 4$ and $8$ dimensions.  Over a field of characteristic $2$ composition algebras exist in $2,4$ and $8$ dimensions.  We can construct composition algebras of  larger dimension from composition algebras of smaller dimension by using the Cayley-Dickson process.  Usually we can choose an orthogonal basis with respect to a bilinear form.  The space $ke$ is not nonsingular when $\ch(k)=2$.  In order to get around this we choose any element $u\in C$ such that $\langle u,e \rangle \neq 0$.  By doing so we can arrive at the following result concerning the basis of an octonion algebra.

\begin{prop}\label{basis}
If $\ch(k) = 2$ (except when $k =\mathbb{F}_2$, $\dim_k(C) = 2$ with $\q$ isotropic) then $C$ has a basis $\{e, u, v, uv, w, uw, vw, (uv)w\}$ with
\[ \langle e,u \rangle = 1, \langle v, uv \rangle = \q(v), \langle w, uw \rangle = \q(w), \langle vw, (uv)w \rangle = \q(v)\q(w) \]
and all the other bilinear forms between distinct pairs of basis vectors are zero.
\end{prop}

Composition algebras have a minimum equation
\begin{equation}\label{mineq}
x^2 + \langle x,e \rangle x + \q(x)e = 0, 
\end{equation}
which is often useful, as is its linearization
\begin{equation}\label{mineqlin}
xy + yx + \langle x,y \rangle e+ \langle x,e \rangle y + \langle y,e \rangle x = 0.
\end{equation}
We also define a map of order $2$ that is analogous to complex conjugation
\[ \bar{x} = \langle x,e \rangle e + x, \]
and note that $x^{-1} = \q(x)^{-1} \bar{x}$ when $\q(x) \neq 0$, and that $x^{-1}$ does not exist otherwise.

\section{Four dimensional subalgebras}\label{4Dsubalgebras}
\label{4dalgs}

When $\ch(k) \neq 2$, an octonion algebra $C$ has only nonsingular subalgebras, which are all quaternion subalgebras.  This is not true when $\ch(k) = 2$; there are also totally singular subalgebras with diagnolizable quadratic forms.  Along with quaternion subalgebras we are concerned with four dimensional totally singular subalgebras of $C$, as they correspond to a type of automorphisms of order $2$ in $\Aut(C)$.

When $k$ is not perfect we have octonion division algebras over fields of characteristic $2$.  In this case their four dimensional subalgebras are all division algebras, and so there are two main classes of of subalgebras: quaternion division algebras and four dimensional totally singular subalgebras.  The latter of these are purely separable field extensions of $k$.  We can also have split octonion algebras over non-perfect fields.  These algebras have both split and division quaternion subalgebras.  Four dimensional totally singular subalgebras can be either split, division, or an intermediate type.  

In general if we take $u \in C$ to be an element such that $\langle u,e \rangle =1$ we can construct a basis for $C$ by taking $v,w$ perpendicular to $u$ and $e$ to get
\[ \{e, u, v, uv, w, uw, vw, (uv)w \}, \]
with a non-singular bilinear form generated by a quadratic form 
\[
[1,\alpha] \perp \beta[1,\alpha] \perp \gamma([1,\alpha] \perp \beta[1,\alpha]).
\]
By taking subsets of this basis that are closed under multiplication we arrive at two main types; the quaternion basis
\[ \{e, u, v, uv \} \]
with a non-singular bilinear form generated by a quadratic form $[1,\alpha] \perp \beta[1,\alpha]$, and the set
\[ \{e, v, w, vw \} \]
with a totally singular bilinear form generated by a quasi-Pfister form $\llangle \beta, \gamma \rrangle$.

When $k$ is a perfect field, any octonion algebra is split.  In this case the quaternion and four dimensional totally singular subalgebras are all split as well.  If  $k$ is not perfect, within each of these main types of four dimensional subalgebras there could be many isomorphism classes.  For a more detailed discussion of quaternion algebras over characteristic $2$ see \cite{cdl15}.  In the following few results we see the importance of these four dimensional subalgebras.

\begin{prop}
Let $t \in \Aut(C)$ such that $t^2=\id$, then $C^t \neq ke$.
\end{prop}

\begin{proof}
Assume $C^t = ke$. Then there exist distinct basis vectors $v, w \in C$ such that $t(v) \neq v$ and $t(w) \neq w$. Since $t^2 = \id$, $t(v) + v$ and  $t(w) + w$ are fixed by $t$ since, for example,
\[
t\big(t(v) + v\big) = t^2(v) + t(v) = v + t(v).
\]
So $t(v) + v,  t(w) + w \in C^t = ke$. Then $t(v) = v + \delta_1e$ and $t(w) = w + \delta_2 e$ for some $\delta_1, \delta_2 \neq 0$.  Observe that 
\[
t( \delta_1^{-1} v + \delta_2^{-1}w) = \delta_1^{-1}(v+\delta_1 e) + \delta_2^{-1}(w + \delta_2 e) = \delta_1^{-1} v + \delta_2^{-1} w,
\]
which contradicts the assumption that $C^t = ke$.
\end{proof}

\begin{lem} \label{mainlem1}
Let $C$ be an octonion algebra over a field $k$ with $\ch(k) = 2$, and let $u \in C$ such that $u \not\in ke$, $\langle u,e \rangle \neq 0$ and $t(u)=u$. Then $t \in \Aut(C)$ and $t^2 = \id$ if and only if there exists a quaternion subalgebra $D \subset C$ such that $t$ fixes $D$ elementwise.
\end{lem}
\begin{proof}
Let $t \in \Aut(C)$ such that $t^2 = \id$.  If $u \not\in ke$ and $u \in C^t$, then there exists $v \in C$ such that
\[ \langle e,v \rangle = \langle u,v \rangle = 0, \]
and $\{ e, u, v, uv \}$ is the basis of a quaternion subalgebra $D \subset C$.  

If $v \in C^t$ then so is $uv$, and $D$ is fixed elementwise by $t$.  There are no other composition subalgebras containing $D$ other than $C$.  If $t$ fixes all of $C$, then $t = \id$, and if not, then $D = C^t$.

If $v \not\in C^t$, notice that 
\[
t(v + t(v)) = t(v) + tt(v) = v + t(v),
\]
and so $v + t(v) \in C^t$.  In this case 
\[
\langle v + t(v), u \rangle = \langle v, u \rangle + \langle t(v), u \rangle = 0 + \langle t(v),t(u) \rangle = \langle v,u \rangle = 0.
\]
Our only question is whether or not $v+t(v)  \in ke + ku$.  Since $\langle u,e \rangle \neq 0$, then 
\[
\langle x_0e+x_1u, v+t(v) \rangle = 0,
\]
so $v + t(v) \in (ke + ku)^{\perp} \subset C \setminus (ke + ku)$.  And $C^t$ must be a quaternion subalgebra as it contains $u$ such that $\langle u,e \rangle \neq 0$. 
\end{proof}

Now we must address the case when $t$ fixes $v \in C$ such that $v \neq e$ and $\langle v,e \rangle = 0$.  We proceed in a similar manner as in \ref{mainlem1}.

\begin{lem} \label{mainlem2}
Let $C$ be an octonion algebra over a field $k$ with $\ch(k) = 2$, and let $v \in C$ such that $v \not\in ke$, $\langle v,e \rangle =0$ and $t(v)=v$. Then $t \in \Aut(C)$ and $t^2 = \id$ if and only if there exists a four dimensional subalgebra $B \subset C$ such that $t$ fixes $B$ elementwise.
\end{lem}
\begin{proof}
Assume that $C^t = ke \oplus kv$ such that $v^2 = \beta e$, and consider $w, u \in C$ such that $w^2= \gamma e$ and $u^2 = u + \alpha e$.   Then $t(u) \neq u$ and $t(w) \neq w$ but $t(u) + u, t(w) + w \in C^t = ke \oplus kv$.  Notice that
\begin{align*}
t(uw) &= (u+m_0e + m_1 v)(w + n_0e + n_1 v) \\
&= uw + (m_0n_0+\beta m_1n_1)e + (m_0n_1 + m_1n_0)v \\
& \ \ \ \ \ \ \ + m_0 w + m_1 vw + n_0 u + n_1 uv.
\end{align*}
Since $t(uw) + uw \in C^t = ke \oplus kv$, $m_0 = m_1 = n_0 = n_1 = 0$, and hence $t(uw) = uw$. Moreover, the fixed subalgebra must be four dimensional; otherwise $t$ fixes an eight dimensional subalgebra and is the identity map.
\end{proof}

This gives us a partial characterization of our inner automorphisms of order $2$, and gives us two \emph{types} of involutions on $\Aut(C)$.  We say an involution is of \emph{type I} if it fixes a quaternion subalgebra of $C$ and of \emph{type II} if it fixes a totally singular four dimensional subalgebra.  

It is also interesting to note that the totally singular four dimensional subalgebras $B$ are Clifford algebras since all elements in $B$ are orthogonal to each other. For $b \in B$, the Cayley equation simplifies to
\[
b^2 + \langle b, e \rangle b + \q(b)e = 0 \Rightarrow b^2 = \q(b).
\]

\section{Involutions fixing a quaternion subalgebra}
\label{typeI}

It was shown in \cite{sv00} that any element $g\in \Aut(C)$ that leaves a quaternion subalgebra $D \subset C$ invariant is of the form
\begin{equation} \label{invD}
g(x+yw) = cxc^{-1} + (p c y c^{-1} )w, 
\end{equation}
where $p$ and $c \in D$, $w \in D^{\perp}$, $\q(w) \neq 0$, $\q(p)=1$, and $\q(c) \neq 0$.  If $t \in \Aut(C)^D$ such that $t^2 = \id$ then
\begin{equation}\label{fixD}
t(x+yw) = x + (ry) w, 
\end{equation}
where $r^2 = e$ and $\q(r)=1$ since $t$ must act as the identity on $D$.

\begin{lem} \label{isoquat}
Let $C$ be an octonion algebra defined over $k$ a field of characteristic $2$.  If $t_1,t_2 \in \Aut(C)$ are of order $2$ and type $1$ then $t_1 \cong t_2$ if and only if there exists $g\in \Aut(C)$ such that $g^{-1}t_1g = t_2$, $g:D_2 \to D_1$ where $D_i = C^{t_i}$, and $g(r_2) = r_1$ (when $r_i$ and $t_i$ are as in \ref{fixD}).
\end{lem}
\begin{proof}
We take $D_1$ and $D_2$ to be the quaternion subalgebras of $C$ fixed by $t_1$ and $t_2$ respectively.  Let $t_1 \cong t_2 \in \Aut(C)$ have order $2$ then there exists $g \in \Aut(C)$ such that $g^{-1} t_1 g = t_2$.  If we let $y \in D_2$, then 
\[ t_2(y) = g^{-1} t_1 g (y). \]
Now $g:C \to C$ is an automorphism, so $g(D_2) \cong D_2 \subset C$.  Let $g(D_2) = D$ and $g(y) = x \in D$, then 
\begin{align*}
t_1g(y) &=  g t_2(y)\\
t_1g(y) &= g(y).
\end{align*}
This tells us $D \subset D_1$.  Since $g\in \Aut(C)$, it must be true that $g(D_2)=D_1$.

If we consider the equation
\[ t_1g(x'+y'w) = gt_2(x'+y'w), \]
we see that
\[ g(x') + r_1g(y')g(w) = g(x') + g(r_2)g(y')g(w), \]
for all $x',y' \in D_2$ and we have $g(r_2) = r_1$.

Let $t_1, t_2 \in \Aut(C)$ such that $t_1^2 = t_2^2 = \id$ and $t_1, t_2 \neq \id$.  We denote by $D_1$ and $D_2$ the quaternion algebras fixed by $t_1$ and $t_2$ respectively.  Choosing $x',y' \in D_2$ so we have $g(x'), g(y') \in D_1$ and $w \in D_2^{\perp}$ such that $\q(w) \neq 0$ we see that
\[ gt_2(x'+y'w) = g(x') + \big(g(y')g(r_2)\big)g(w), \]
and
\[ t_1g(x'+y'w) = g(x') + \big(g(y') r_1g(w) \big). \]
Restricting to the case that $x'=0, y'=e$ and $g(r_2)=r_1$. Noting that $w$ is invertible we see $gt_2 = t_1g$ for all $x'+y'w \in C$.
\end{proof}

\begin{prop} \label{class}
If $t_1, t_2 \in \Aut(C)$ are of order $2$and type $1$ with $D_1, D_2 \subset C$ their respective quaternion fixed point subalgebras, then the following statements are equivalent;
\begin{enumerate}[$(1)$]
\item $t_1 \cong t_2$ 
\item $\mathcal{I}_{t_1} \cong \mathcal{I}_{t_2}$ 
\item There exists $g \in \Aut(C)$ such that $g(D_2)= D_1$ and $g(r_2) = r_1$.
\end{enumerate}
\end{prop}
\begin{proof}
Notice that 1 implies 2 is concluded almost immediately by noticing that if $gt_1g^{-1} = t_2$, then $\mathcal{I}_{g} \mathcal{I}_{t_1} \mathcal{I}_{g}^{-1} = \mathcal{I}_{t_2}$.  We can arrive at the fact that 2 implies 1 by noticing the center of $\Aut(C)$ is trivial, therefore any inner automorphism of order $2$ is induced by an element of order $2$ in $\Aut(C)$. Then we see that when  $\mathcal{I}_{g} \mathcal{I}_{t_1} \mathcal{I}_{g}^{-1} = \mathcal{I}_{t_2}$ we have
\[ \mathcal{I}_{g} \mathcal{I}_{t_1} \mathcal{I}_{g}^{-1}(h) = \mathcal{I}_{t_2}(h), \]
for all $h \in \Aut(C)$.  This puts $t_2gt_1g^{-1} \in Z(\Aut(C)) = \{\id\}$, and it follows that $gt_1g^{-1} = t_2$.  The results from \ref{isoquat} gives us (1) if and only if (3).  
\end{proof}

In the case $C^t \subset C$ is a split quaternion subalgebra, all $r \in C^t$ such that $r^2 = e$, and $\q(r) =1$ are isomorphic as shown in the following results.  We show in \ref{quatdivfix} that there are no elements of order $2$ in $\Aut(C)$ fixing a division quaternion algebra, and so \ref{class} (3) can be shown to depend on the isomorphism class of the fixed quaternion subalgebra of $C$ only.

\begin{lem}\label{order2mat}
If $R^2 = \id \in \GL_2(k)$ then 
\[ R = \begin{bmatrix} R_{11} & R_{12} \\
					(R_{11}^2+1)R_{12}^{-1} & R_{11}
					\end{bmatrix}, \]
up to conjugation in $\GL_2(k)$.
\end{lem}
\begin{proof}
By straight forward computation $R^2 = \id$ gives us one of the following forms 
\[ \begin{bmatrix} R_{11} & R_{12} \\
					(R_{11}^2+1)R_{12}^{-1} & R_{11}
					\end{bmatrix} \text{ or }
\begin{bmatrix} R_{11} & (R_{11}^2+1)R_{12}^{-1}  \\
					R_{12} & R_{11}
					\end{bmatrix}, \]
which are in the same conjugacy class through conjugation by $\left[\begin{smallmatrix} 0 & 1 \\
									1 & 0 \\
									\end{smallmatrix}\right]$.
\end{proof}

\begin{thm}
Let $D$ be a split subalgebra of $C$ and $r \in D$ with $r^2=e$ then
\[ r \cong \begin{bmatrix}
			1 & 1 \\
			0 & 1 \\
			\end{bmatrix} \in \Mat_2(k). \]
\end{thm}
\begin{proof}
Recall that when $D$ is a split quaternion algebra $D \cong \Mat_2(k)$.  Let $r$ correspond to $R \in \Mat_2(k)$.  Then take $P \in \GL_2(k)$ and compute $\mathcal{I}_P(R)$ where $R^2=\id$.  Then $R$ is of the form in \ref{order2mat} where $R_{12} \neq 0$.  Now we see that taking
\[ P= \begin{bmatrix}
		1 & 0 \\
		R_{11}+1 & R_{12}
		\end{bmatrix}, \]
we have $\mathcal{I}_P(R) = \left[\begin{smallmatrix} 1 & 1 \\
									0 & 1 \\
									\end{smallmatrix}\right]$.
\end{proof}

\begin{cor} \label{sl2conj}
 If $r_1, r_2 \in \SL_2(k)$, $r_1^2=r_2^2=\id$ with $\ch(k)=2$, then $r_1$ is conjugate to $r_2$.
 \end{cor}

When considering $\Aut(C)$ over a field of characteristic $2$, where $k^2 \neq k$, we still need to consider quaternion subalgebras of $C$ as the result from \ref{class} still holds.  

When $k$ is not perfect and $t$ fixes a split quaternion subalgebra of a split octonion algebra, the conjugacy class of $t$ only depends on the isomorphism class of the quaternion subalgebra $D$ fixed by $t$.

If $D$ is a quaternion algebra it has a basis of the form
\[ \{ e, u, v, uv \}, \]
where $u^2 = u+ \alpha e$, $v^2 = \beta e$ and $vu = uv + v$ with $\alpha \in k$ and $\beta \in G_*(k)$ see \cite{dr83}.  When this is the case our quadratic form is as follows
\[ \q(c) = c_0^2 + c_0c_1 + \alpha c_1^2 + \beta( c_2^2 + c_2c_3 + \alpha c_3^2), \]
which implies that $q \sim  [1,\alpha] \perp \beta[1,\alpha].$
We are interested in maps of the form $t(x+yw) = x+(ry)w$ where $\q(r)=1$, $r^2=e$ and $r \in C^t$.  If we consider 
\[ r = \rho_0 e + \rho_1u + \rho_2 v + \rho_3 uv, \]
and $r^2 = e$ using the basis above we see that
\[ e = (\rho_0^2 + \alpha \rho_1^2 + \beta \rho_2^2 + \beta \rho_2\rho_3 + \alpha \beta \rho_3^2)e + \rho_1^2u + \rho_1\rho_2 v + \rho_1\rho_3 uv, \]
and so $\rho_1=0$ and $\rho_0^2 + \beta (\rho_2^2 +  \rho_2\rho_3 + \alpha \rho_3^2)=1$.  Notice that this is the norm of $r$ and $\q(r) = 1$, which must be true since $\ch(k) = 2$ and $\q(r)^2 = 1$. 

\begin{prop}\label{noquatdiv}
If there exists $r\in D$ a quaternion algebra such that $r \neq e$ and $r^2 =e$, then $D$ is split. 
\end{prop}
\begin{proof}
If $r \in D$ and $r^2 = e$, then 
\[ \q(r) = \rho_0^2 + \beta(\rho_2^2 + \rho_2\rho_3 + \alpha \rho_3^2), \]
and $\q(r) = 1$.  If $\q(r) = 1$, then
\begin{align*}
1  &= \rho_0^2 + \beta(\rho_2^2 + \rho_2\rho_3 + \alpha \rho_3^2) \\
0 &= (1+\rho_0)^2 + \beta(\rho_2^2 + \rho_2\rho_3 + \alpha \rho_3^2).
\end{align*}
This tells us that $\q(e+r) = 0$, which tells us $r+e$ is a nontrivial zero divisor when $r \neq e$.
\end{proof}

\begin{cor} \label{quatdivfix}
There are no automorphisms of order $2$ fixing a quaternion division algebra.
\end{cor}

In the case that $C$ is a division algebra and $t\in \Aut(C)$ such that $t^2 = \id$, $t$ must fix a totally singular subalgebra.  In which cases the $\Aut(C)$-conjugacy class of $\Aut(C)^{\mathcal{I}_t}$ depends solely on the conjugacy class of the totally singular subalgebra.  So \ref{class} can be restated as the following.

\begin{lem} \label{class2}
If $t_1, t_2 \in \Aut(C)$ are of order $2$ with $D_1=C^{t_1}, D_2=C^{t_2}$, then the following statements are equivalent;
\begin{enumerate}[$(1)$]
\item $t_1 \cong t_2$ 
\item $\mathcal{I}_{t_1} \cong \mathcal{I}_{t_2}$ 
\item $D_2 \cong D_1$.
\end{enumerate}
\end{lem}

It is known, see \cite{sv00}, that all quaternion and octonion algebras are split when taken over a perfect field of characteristic $2$.  This follows from the fact that $k^2=k$, and by noticing that for any choice of a quadratic form there exists a nonzero element $x \in D \subset C$ such that $\q(x)=0$.  So when $k$ is a finite field or algebraically closed we see that the results follow directly from our results above.  Further there are no involutions fixing division quaternion algebras.

\begin{cor}\label{fixD}
If $D$ is a split quaternion subalgebras of C over a field of characteristic $2$ there is only one isomorphism class of inner $k$-involutions fixing $D$.
\end{cor}
\begin{proof}
By \ref{class} we know that the classes of $k$-involutions of $\Aut(C)$ are in bijection with the isomorphism classes of quaternion subalgebras of $C$.
\end{proof}

\begin{prop}
If $g \in \Aut(C)$ and $t\in \Aut(C)$ such that $t^2 = \id$ is of the form described in \ref{fixD}, then $g \in \Aut(C)^{\mathcal{I}_t}$ if and only if $g$ leaves $D$ invariant and $rpc = pcr$ .
\end{prop}
\begin{proof}
Let $t^2 = \id$, then by \ref{class} we have $C^t = D \subset C$ is a quaternion subalgebra.  By \cite{sv00} we can decompose $C = D \oplus Dw$ where $\q(w) \neq 0$, $w \in D^{\perp}$.  So now we can write any element in $C$ in the form $x+yw$, where $x,y \in D$.  Let $g \in \Aut(C)$ such that $g(D) = D$, then by \cite{sv00}
\[ g(x+yw) = cxc^{-1} + (pcyc^{-1}) w, \]
where $\q(p)=1$, and
\[ t(x+ yw) =  x +  (ry)w, \]
for $\q(r) = 1$.  We can compute
\begin{align*}
gt(x+yw) &= g(x + (ry)w) \\
&= cxc^{-1} +  (pcryc^{-1}) w \\
&= cxc^{-1} +  (rpcyc^{-1}) w \\
&= tg(x+yw).
\end{align*}
for all $x,y \in D$, so $g \in \Aut(C)^{\mathcal{I}_t}$.  To finish the proof let us assume $g \in \Aut(C)^{\mathcal{I}_t}$.  Let $x,y \in D$ as above
\begin{align*}
g(x+yw) &= tgt(x+yw) \\
&= tg(x) + t(g(ry)g(w)) \\
g(x)+g(y)g(w) &= tg(x) + t(g(ry)g(w)).
\end{align*}
Letting $y=0$ we see that $g(x) = tg(x)$, so $g(x) \in C^t=D$ for all $x\in D$.  Since $g$ is a bijection, $g(D) = D$.
\end{proof}

Notice this is different from the analogous result when $\ch(k) \neq 2$ in that case we can take $r = - e$ and there are no commutativity issues.  When $\ch(k) \neq 2$ and $t$ fixes a split quaternion subalgebra of $C$ the fixed point group is of the form
\[ \Aut(C)^{\mathcal{I}_t} \cong \Aut(D) \ltimes \Sp(1), \]
where $\Sp(1) = \{ p \in C \ | \ \q(p)=1 \}$.  For split quaternion algebras over a given field there is only one isomorphism class.  The correspondence between quaternion subalgebras of $C$ and certain conjugacy classes of $k$-involutions still holds when $\ch(k) = 2$, but the fixed point group of $\mathcal{I}_t$ takes a different form.

\begin{thm}
\label{fixpt1}
For $k$ a field with $\ch(k) = 2$, when $t \in \Aut(C)$ is an involution of type I
\[ \Aut(C)^{\mathcal{I}_t} \cong \SL_2(k) \times G_+(k). \]
\end{thm}
\begin{proof}
Let $k$ be a field of characteristic $2$.  To construct an automorphism of order $2$ of $C$, we need only to choose an $r\in D$ with $r^2 = e$ and $\q(r) =1$.  Since there is only one isomorphism class of quaternion subalgebras in this case we have only one isomorphism class of automorphisms of order $2$.  Choose $D = \Mat_2(k)$, $t(x+yw) = x + (ry)w$ with
\[ r = \begin{bmatrix}
		1 & 1 \\
		0 & 1 \\
		\end{bmatrix}. \]
If we take $g \in \Aut(C)^{\mathcal{I}_t}$, $g$ must be of the form
\[ g_{(c,p)}(x+yw) = cxc^{-1} + (p c y c^{-1})w, \]
where $\q(c)\neq 0$, and $p \in \Sp(1)$.  Since $p$ is invertible $pc \in \GL_2(k)$.  Simple calculations show that if $pcr=rpc$, we have
\[ pc \in UT(k) = \left\{ \begin{bmatrix}
				a_0 & a_1 \\
				0 & a_0
				\end{bmatrix} \ \bigg| \ a_0^2 = \q(c), a_1 \in G_+(k) \right\}. \]
We now consider the map $g_{(c,p)} \mapsto g_{(c,P)}$ where $P = pc$ for some $p \in \Sp(1)$.  In this case $g$ takes the form
\begin{align*}
 g_{(c,P)}(x+yw) &= cxc^{-1} + \left(\begin{bmatrix}
				a_0 & a_1 \\
				0 & a_0
				\end{bmatrix}yc^{-1}\right)w \\
&= \left(a_0^{-1}c\right)x\left(a_0^{-1}c\right)^{-1} + \left(\begin{bmatrix}
				1 & a_0^{-1}a_1 \\
				0 & 1
				\end{bmatrix}y\left(a_0^{-1}c\right)^{-1}\right)w \\
&= g_{(a_0^{-1}c, a_0^{-1}P)}(x+yw),
\end{align*}
when $\det(c) = a_0^2 = \q(c)$.  We take
\[ \check{p} \in G_+(k) = \left\{\begin{bmatrix}1 & \check{a_1} \\
				0 & 1\end{bmatrix} \ \bigg| \ \check{a_1} \in k \right\}. \] 
Since $c\in \GL_2(k)$ we have $a_0^{-1} c \in \SL_2(k)$.  The surjective homomorphism 
\[ \Psi:\SL_2(k) \times G_+(k) \to \Aut(C)^{\mathcal{I}_t} \]
\[ \ \ \ \ \ \  \ \ \ \ \ \ (a_0^{-1}c,\check{p}) \mapsto g_{(a_0^{-1}c,\check{p})}, \]
has $\ker(\Psi) = Z(\SL_2(k)) \times \{ \id \}=\{ \id \} \times \{ \id \}$ since we are in a field of characteristic $2$, and we are left with 
\[ \Aut(C)^{\mathcal{I}_t} \cong \SL_2(k) \times G_+(k). \]
\end{proof}

\begin{thm}
If $C$ contains a split quaternion subalgebra there is one $\Aut(C)$-conjugacy class of involutions of type I, otherwise there are no involutions of type I.  Moreover, if $t$ is an involution of type I, then the fixed point group of $t$ is $\Aut(C)$-conjugate to $\SL_2(k) \times G_+(k)$.
\end{thm}

\begin{proof}
By \ref{class2} we now that $\Aut(C)$-conjugacy classes of involutions of type I are in bijection with $\Aut(C)$-conjugacy classes of fixed point groups of involutions of type I, and that the $\Aut(C)$-conjugacy classes of fixed point groups are in bijection with isomorphism classes of quaternion subalgebras of $C$.  And then finally we note that there is only one isomorphism class of split quaternion subalgebras and that there are no involutions of type I fixing a division quaternion subalgebra by \ref{quatdivfix}.  Finally \ref{fixpt1} gives us a representative of the $\Aut(C)$-conjugacy class of the fixed point group.
\end{proof}

In the next section we explore a similar characterization for four dimensional totally singular subalgebras of an octonion algebra and involutions of type II.

\section{Involutions fixing a four dimensional totally singular subalgebra}
\label{typeII}

Let $B$ be a four dimensional totally singular subalagebra of an octonion algebra $C$. In this section, we determine the conjugacy classes of automorphisms of order $2$ that fix $B$. From \ref{4Dsubalgebras}, we know that $B$ has basis $\{e, v, w, vw \}$ with
\[ v^2 = \beta e \text{ and }w^2 = \gamma e. \]
There exists some $u \in C \setminus B$ with $\bi{u,e} = 1$ and $\q(u) = \alpha \in k$. It follows that $\bi{v, vu} = \beta$, $\bi{w, wu} = \gamma$, and $\bi{vw, (vw)u} = \beta\gamma$, and the bilinear form of all other pairs of basis vectors is zero. The basis of $B$ can be extended to a basis $\{e, v, w, vw, u, vu, wu, (vw)u\}$ of $C$ by noting that $C = B \oplus Bu$. The quadratic form $\q$ on $B$ is a totally singular quasi-Pfister form: $\q \cong \llangle \beta, \gamma \rrangle = \langle 1, \beta, \gamma, \beta\gamma\rangle$, for some $\beta, \gamma \in G_*(k)$. For $x \in B$, we have
\[ 
\q(x) = x_0^2 + \beta x_1^2 + \gamma x_2^2 + \beta\gamma x_3^2. 
\]

The following Corollary is a direct result of \ref{noquatdiv}, \ref{mainlem1} and \ref{mainlem2}.  

\begin{cor}
When $C$ is a division algebra, all automorphisms of order $2$ fix 4 dimensional totally singular subalgebras $B$ of $C$. Moreover, $B$ is also a subfield of $C$.
\end{cor}

We use \ref{bar} to prove \ref{z} below.

\begin{lem}\label{bar}
If $g \in \Aut(C)$ and $x \in C$, then $\overline{g(x)} = g(\bar{x})$.
\end{lem}

\begin{proof}
Since $g$ is an automorphism of $C$, $g$ is an isometry \cite{sv00}; thus $g$ preserves the bilinear form on $C$. For any $x \in C$, $\bar{x} = \bi{x,e}e + x$. Observe that
\begin{align*}
g(\bar{x}) &= g\big(\bi{x,e}e + x\big)\\
&= g\big(\bi{x,e}e\big) + g(x)\\
&= \big\langle g(x), g(e)\big\rangle g(e) + g(x)\\
&= \big\langle g(x), e\big\rangle e + g(x)\\
&= \overline{g(x)}.
\end{align*}
\end{proof}

\begin{prop}\label{z}
If $g \in \Aut(C)$, then $g(u) = u + m + nu$ for some $m + nu \in ke^{\perp}$, where $m \in B$ $n \in kv \oplus kw \oplus kvw$ and $\q(m+nu) = \langle m, u \rangle$.
\end{prop}

\begin{proof}
Suppose $g(u) = u + m + nu$, where $m,n \in B$. Let
\[ 
m = m_0e + m_1v + m_2w + m_3vw, \text{ and } \ n = n_0e + n_1v + n_2w + n_3vw
\]
On one hand, $\overline{g(u)} = n_0 e + g(u)$. On the other hand, $g(\bar{u}) = e + g(u)$. By \ref{bar}, $n_0 = 1$. It follows that $\bi{m + nu, e} = 0$, and hence $m + nu \in ke^{\perp}$.

Since $u^2 = u + \alpha e$, we have
\[
g(u)^2 = g(u^2) = g(u + \alpha e) = g(u) + \alpha e.
\]
Letting $z = m + nu$, we have
\begin{align*}
g(u)^2 &= g(u) + \alpha e \\
(u+z)^2 &= u+z + \alpha e \\
u^2 + uz + zu + z^2 &= u + z + \alpha e \\
u + \alpha e + uz + uz + \langle u,z \rangle e + \langle u,e \rangle z + \langle z,e \rangle u + z^2 &= u + z + \alpha e \\
u + \alpha e + \langle u,z \rangle e + z + \q(z)e &= u + z + \alpha e,
\end{align*}
by applying \ref{mineq} and \ref{mineqlin} to $z^2$ and $zu$.  Collecting terms, we have $\q(z) = \bi{u, z} = \bi{u, m+nu} = m_0 = \langle m,u \rangle$.
\end{proof}

Next we state some definitions and notation we use in the remainder of this section. Define the set 
\begin{equation}\label{bhatplus}
\widehat{C} = \{ x + yu \in ke^{\perp} \ | \ \q(x + yu) =   x_0 \},
\end{equation}
where $x, y \in B$, which is a groups under the product induced by composition of functions in $\Aut(C,B)$. We also define
\[
\widehat{B} = \widehat{C} \cap B,
\]
which is a group under addition.  We will often refer to
\[
\widetilde{B} = \{ b \in B \ | \ \langle b,u \rangle = 0 \ \} = kv \oplus kw \oplus kvw, \]
and to this end we define the projection map \  $\widetilde{ \ }: B \to \widetilde{B}$ to be defined by $\tilde{b} = b + \langle b,u \rangle e$. Notice that $y \in \widetilde{B}$ in \ref{bhatplus}.



We now consider isomorphism classes of $\widehat{B}$ for different classes of $B$.  Taking $\q(b) = \langle b,u \rangle = b_0$ we have
\[
b_0^2 + \beta b_1^2 + \gamma b_2^2 + \beta \gamma b_3^2 = b_0 \iff  \beta b_1^2 + \gamma b_2^2 + \beta \gamma b_3^2 = b_0^2 + b_0 \in \wp(k).
\]
In other words, elements in $\widehat{B}$ correspond to elements in $\wp(k) \cap \q(\widetilde{B})$.

\begin{prop}
When $B$ is a split totally singular four dimensional subalgebra of $C$, then 
\[
\widehat{B} \cong G_+(k^2) \times G_+(k) \times G_+(k).
\]
\end{prop}

\begin{proof}
If $B$ has quadratic form isometric to $\llangle 1, 1 \rrangle$ then $B$  is split. We have the equation
\[
b_0 = b_0^2 + b_1^2 + b_2^2 + b_3^2 = (b_0 + b_1 + b_2 + b_3)^2. 
\]
This implies $b_0 \in k^2$, since $b_3$ depends on our choice of $b_0, b_1$, and $b_2$.  So there is an isomorphism of additive groups
\[ \widehat{B} \cong G_+(k^2) \times G_+(k) \times G_+(k). \]
\end{proof}

The additive group $G_+(k^2)$ is often denoted $\alpha_2(k)$ and is the additive group of squares in $k$ when $\ch(k) = 2$.  

\begin{cor}
When $k$ is perfect, $k^2 =k$ and $B$ is as above we have
\[ \widehat{B} \cong G_+(k) \times G_+(k) \times G_+(k). \]
\end{cor}

As for the non-split four dimensional totally singular subalgebra $B$, we appeal to the theory of quadratic forms.  In these cases the quadratic forms are quasi-Pfister forms where
\[
q(z) = z_0^2 + \beta z_1^2 + \gamma z_2^2 + \beta\gamma z_3^2,
\]
and we write $q  = \llangle \beta, \gamma \rrangle = \langle 1, \beta, \gamma, \beta\gamma \rangle$.  The following result due to Hoffman and Laghribi can be found in \cite{hl04}; an equivalent result also appears in the discussion of totally singular quadratic forms in \cite{ekm08}.

\begin{prop}\label{hofflag}
There exists a natural bijection between anisotropic $n$-fold quasi-Pfister forms and purely inseparable extensions of $k^2$ of degree $2^n$ inside $k$ which is given by
\[ \llangle a_1,a_2, \ldots, a_n \rrangle \leftrightarrow k^2( a_1,a_2, \ldots, a_n). \]
\end{prop} 

\begin{cor} \label{k2field}
Let $B$ be a totally singular four dimensional subalgebra and $\beta, \gamma \in G_*(k)$. Then
\begin{enumerate}[$(1)$]
\item \label{Bdiv} $[k^2(\beta,\gamma): k^2]=4$ if and only if $\llangle \beta, \gamma \rrangle$ defines a division algebra 
\item\label{Bint} $[k^2(\beta,\gamma): k^2]=2$ if and only if $\llangle \beta, \gamma \rrangle$ an algebra of indeterminate type (with a $2$ dimensional division subalgebra)
\item\label{Bspl} $[k^2(\beta,\gamma): k^2]=1$ if and only if $\llangle \beta, \gamma \rrangle$ defines a split algebra.
\end{enumerate}
\end{cor}

Notice that all of our division subalgebras of type $B$ are fields since this subalgebra is commutative.  Let us consider some specific examples of algebras of type $B$ over fields of rational functions with coefficients in a perfect field.  

\begin{ex}\normalfont
Let $\mathbb{F}$ be any finite (or algebraically closed) field of characteristic $2$, and consider $k = \mathbb{F}(\chi_1)$. In this case only split algebras or those of indeterminate type occur. Since $k^2 = \mathbb{F}(\chi_1^2)$, if $\beta \in G_*(k)\backslash G_*(k)^2$ then $k^2(\beta) = k$.  So the largest $k^2$-dimension of a field of the form $k^2(\beta, \gamma)$ is $2$, and we can choose $\beta, \gamma \in k^2$ to get an algebra of split type.  

When we adjoin another indeterminate, we have $k=\mathbb{F}(\chi_1,\chi_2)$, and we now have enough degrees of freedom to construct a four dimensional division algebra.  The squares in $k$ make up the subfield $k^2=\mathbb{F}(\chi_1^2, \chi_2^2)$.  There is only one division algebra since 
\[
k= k^2 \oplus k^2 \chi_1 \oplus k^2 \chi_2 \oplus k^2\chi_1\chi_2,
\]
and so $[k: k^2] = 4$, and $[k^2(\beta,\gamma): k^2]\leq 4$ when $\beta, \gamma \in k$.
\end{ex}

\begin{prop} \label{tinduce}
If $t\in \Aut(C)$ has order $2$ and fixes a totally singular subaglebra of $C$, then
\[ t(x+yu) = x + y(u+b), \]
for some $b \in \widehat{B}$.
\end{prop}

\begin{proof}
By \ref{z}, we have $t(u) = u + m + n u$ with $m + nu \in ke^{\perp}$ and $\q(m + nu) = m_0$.  It remains to show that $m + nu \in B$.  Observe that
\[
t^2(u) = t(u + m + nu)  =  u + m + nu + t(m + nu),
\]
and since $t^2 =\id$ we have
\[
u + m + nu + t(m + nu) = u,
\]
and so $t(m + nu) = m + nu$. And we have $m+nu=b \in \widehat{B}$.
\end{proof}

\begin{prop}\label{gfixB}
Let $t \in \Aut(C)$ with $t^2 = \id$ and fixed point subalgebra $B$.  If $g \in \Aut(C)^{\mathcal{I}_t}$, then $g(B) = B$.
\end{prop}

\begin{proof}
For any $x\in B$ and $g \in \Aut(C)^{\mathcal{I}_t}$, we have $tg(x) = gt(x)  = g(x)$. This shows that $g(x)$ is fixed by $t$, hence $g(x) \in B$. On the other hand, since $g$ and $t$ are bijections, the other containment holds as well.
\end{proof}

\begin{prop}  
Denote by $\Aut(C)^B$ the automorphisms of $C$ that fix $B$ elementwise. Then 
\[ 
\Aut(C)^B \cong  \widehat{B}. 
\]
\end{prop}

\begin{proof}
Let $x + yu \in B \oplus Bu = C$ and let $g \in \Aut(C)$ such that $g$ fixes $B$ elementwise.  Then by \ref{z} we have that $g(u) = u + m + nu$ where $m \in B$, $n \in \widetilde{B}$ and $\q(m+nu) = \langle m,u \rangle$.  Now, since $g \in \Aut(C) \subset \Orth(\q)$, for all $a \in B$ we have that 
\[
\langle a, u \rangle = \langle g(a), g(u) \rangle = \langle a, u + m+ nu \rangle = \langle a, u \rangle + \langle \tilde{a}, nu \rangle.
\]
This gives us $\langle \tilde{a}, nu \rangle = 0$ for all $\tilde{a} \in \widetilde{B}$, which is only true if $n=0$.  So $g(x+yu) = x + y(u+m)$ where $m\in \widehat{B}$.  Moreover $g$ is an involution of type II.
\end{proof}

\begin{cor}
An automorphism $g$ is an involution of type II if and only if $g \in \Aut(C)^B$.
\end{cor}

\begin{prop} \label{Chatgroup1}
Let $g \in \Aut(C,B)$ such that $g(u) = u + m + nu$. Then
\[
\langle g(a), u \rangle = \langle a,u \rangle + \langle g(a), nu \rangle.
\]
\end{prop}
\begin{proof}
If $g \in \Aut(C,B)$, then $g \in \Orth(\q)$. Then
\begin{align*}
\langle a,u \rangle &= \langle g(a),g(u) \rangle \\
&= \langle g(a), u + m + nu \rangle \\
&= \langle g(a), u \rangle + \langle g(a), nu \rangle,
\end{align*}
since $B \subset B^{\perp}$.  Solving for $\langle g(a), u \rangle$ we have our relation.
\end{proof}

\begin{prop}\label{typeIIiso}
If $t$ and $s$ are involutions of type II fixing a subalgebra $B$, where $t(u) = u+ b$ and $s(u) = u + b'$, then $t$ and $s$ are $\Aut(C)$-conjugate if and only if there exists $g \in \Aut(C,B)$ with $g(u) = u + m + nu$, $m + nu \in \widehat{C}$, and $g(b) = (e + n)b'$.
\end{prop}
\begin{proof}
Let $g \in \Aut(C)$ such that $gtg^{-1} = s$. Then $g(u)= u + m + nu$ for $m + nu \in \widehat{C}$, and
\[ gt(u) = g(u + b) = u + m + nu + g(b),  \]
and
\[ sg(u) = s(u + m + nu) = u + b' + m + ns(u) = u + b' + m+nu + nb'. \]
Setting these two expressions equal to each other gives us 
\[ g(b) = (e+n)b'. \]
  If we assume there exists $g\in \Aut(C)$ such that $g(b) = (e+n)b'$ we can reverse the argument to show the converse.
\end{proof}

\begin{lem} \label{Btildeinv}
If $g \in \Aut(C)^{\mathcal{I}_t}$ where $t(u) = u + e$, then $g\big(\widetilde{B}\big) = \widetilde{B}$.
\end{lem}
\begin{proof}
Let $g \in \Aut(C)^{\mathcal{I}_t}$. Then $g(B) = B$, and $g(u) = u + m$ for $m \in \widehat{B}$.  For all $a \in B$, recall that $\tilde{a}$ is the projection of $a$ onto $\widetilde{B}$. Then we have
\[ \langle a,u \rangle = \langle g(a), g(u) \rangle = \langle g(a), u + m \rangle = \langle g(a),u \rangle.\]
Now $\langle g(a) , u \rangle = \langle \langle a,u \rangle e + g(\tilde{a}), u \rangle$, so we see that
\begin{align*}
\langle a,u \rangle &= \langle \langle a,u \rangle e + g(\tilde{a}), u \rangle \\
\langle a,u \rangle &= \langle a,u \rangle \langle e,u \rangle + \langle g(\tilde{a}), u \rangle \\ 
0 &= \langle g(\tilde{a}),u \rangle, 
\end{align*}
for all $a \in B$.  So $g(\tilde{a}) \in \widetilde{B}$. 
\end{proof}

\begin{prop}
Let $t(u)=u + e$. Then $\Aut(C)^{\mathcal{I}_t} \cong \Aut(B) \ltimes \widehat{B}$.
\end{prop}
\begin{proof}
Let $t$ be an involution of type II such that $t(u) = u + e$. Now, define the map $\Psi: \Aut(C)^{\mathcal{I}_t} \to \Aut(B) \ltimes \widehat{B}$ to be the map
\[ \Psi(g) = (g_B, m), \]
where $g_B = g|_B$ and $g(u) = u + m$.  By \ref{Btildeinv} and through straight forward computation, it can be shown that every $g_B \in \Aut(B)$ can be extended to an element $(g_B, m) \in \Aut(C)^{\mathcal{I}_t} \subset \Aut(C)$ for any $m \in \widehat{B}$.  So $\Psi$ is onto, and the $\ker(\Psi)$ is the set of automorphisms leaving $B$ invariant that restrict to the identity on $B$ and map $u$ to itself or in other words $\ker(\Psi) = \{\id\}$.  We can also show that $\Psi$ is a homomorphism by computing
\begin{align*}
g'g(x+yu) &= g' \big(g_B(x) + g_B(y) (u + m) \big) \\
&= g'_Bg_B(x) + g'_Bg_B(y) g'_B(u + m) \\
&= g'_Bg_B(x) + g'_Bg_B(y)\big(u + m' + g'_B(m) \big)
\end{align*}
and taking the product in $\Aut(B) \ltimes \widehat{B}$ to be 
\[ (g'_B,m')(g_B,m) = \big(g'_Bg_B,m'+ g'_B(m) \big) \Rightarrow \Psi(g')\Psi(g) = \Psi(g'g). \]
\end{proof}

\begin{lem}\label{isofixpoint}
Let $t$ and $s$ be involutions of type II fixing the same subalgebra $B$. If $t$ is $\Aut(C)$-conjugate to $s$, then $\Aut(C)^{\mathcal{I}_t}$ is $\Aut(C)$-conjugate to $\Aut(C)^{\mathcal{I}_s}$.
\end{lem}
\begin{proof}
Consider $h \in \Aut(C)^{\mathcal{I}_s}$ and $gtg^{-1} = s$.   This requires that $g(B)=B$.  Then 
\[ h = shs = gtg^{-1} h gtg^{-1}, \]
and 
\[ g^{-1}h g = t g^{-1} h g t. \]
\end{proof}

\begin{lem}\label{ycon}
If $C=B\oplus Bu$ is not a division algebra and $b \in \widehat{B}$ with $q(b) \neq 0$ then there exists $n \in \widetilde{B}$ such that $n = b^{-1} + e$.
\end{lem}
\begin{proof}
For any $b \in \widehat{B}$,  $b^{-1}= \q(b)^{-1}b$, and the coefficient of the identity component of $b$ is $\q(b)$ by definition of $\widehat{B}$. So $b^{-1} = e + \q(b)^{-1}\tilde{b}$.  So we have $n=\q(b)^{-1}\tilde{b} \in \widetilde{B}$.
\end{proof}

\begin{cor}\label{conjtoIDqbne0}
Let $t$ be an involution of type II fixing a totally singular subalgebra $B$ such that $t(u) = u + b'$ and $\q(b') \neq 0$. Then $t$ is $\Aut(C)$-conjugate to the map $u \to u + e = \bar{u}$.
\end{cor}
\begin{proof}
By choosing $b = e$ in \ref{typeIIiso} we have $e = (e+n)b'$.  We are now in the situation described by \ref{ycon}, and so by \ref{typeIIiso} we have that the two maps are $\Aut(C)$-conjugate.
\end{proof}

\begin{lem}\label{conjtoIDqb0}
Let $B$ be a totally singular subalgebra that is not a division algebra. Let $s$ be the involution of type II such that $s(u) = u + e$, and let $t$ be an involution of type II such that $t(u) = u + b$ with $\q(b) = 0$ and $b\neq 0$. Then $t$ is $\Aut(C)$-conjugate to $s$.
\end{lem}

\begin{proof}
Suppose $\q(b) = 0$. We want to find some $g \in \Aut(C,B)$ such that $gt = sg$. By \ref{typeIIiso}, we have that $g(b) = e + n$, where $\q(n) = 1$ and $n \in \tilde{B}$. Since $b \in \widehat{B}$, then $b = b_0 e + b_1v + b_2w + b_3vw$, and since $\q(b) = 0$ it follows that $b_0 = 0$. Since $g$ is an automorphism, then $g$ is a bijection on $B$, and $g$ preserves the norm and the bilinear form.

There are two possible isometry classes of quadratic forms $\llangle \beta, \gamma \rrangle$ where $B$ is not a division algebra $\llangle 1, \gamma \rrangle$ and $\llangle 1,1 \rrangle$.
Since $\q(b) = 0$ and $b \in \widehat{B}$, we have
\begin{equation}
\beta b_1^2 + \gamma b_2^2 + \beta\gamma b_3^2 = 0.
\end{equation}
Since $g(e) = e$, we can write the action of $g$ on the basis of $B$:
\begin{align*}
g(e) &= e\\
g(v) &= a_0e + a_1v + a_2w + a_3vw\\
g(w) &= d_0e + d_1v + d_2w + d_3vw\\
g(vw) &= g(v)g(w)\\
&= (a_0d_0 + \beta a_1d_1 + \gamma a_2d_2 + \beta\gamma a_3d_3) e \\
& \ + (a_0d_1 + a_1d_0 + \gamma a_2d_3 + \gamma a_3d_2) v\\
& \ + (a_0d_2 + a_2d_0 + \beta a_1d_3 + \beta a_3d_1) w\\
& \ + (a_0d_3 + a_3d_0 + a_1d_2 + a_2d_1) vw
\end{align*}
where $g(v)$ and $g(w)$ determine the image of $g(vw)$. The images are further subject to the norm conditions:
\begin{align}
a_0^2 + \beta a_1^2 + \gamma a_2^2 + \beta\gamma a_3^2 &= \beta\\
d_0^2 + \beta d_1^2 + \gamma d_2^2 + \beta\gamma d_3^2 &= \gamma.
\end{align}
Since $g$ is a bijection, we have $\det(M) \ne 0$, where
\[
M = 
\begin{bmatrix}
1 & a_0 & d_0 & a_0d_0 + a_1d_1 + a_2d_2 + a_3d_3\\
0 & a_1 & d_1 & a_0d_1 + a_1d_0 + a_2d_3 + a_3d_2\\
0 & a_2 & d_2 & a_0d_2 + a_2d_0 + a_1d_3 + a_3d_1\\
0 & a_3 & d_3 & a_0d_3 + a_1d_2 + a_2d_1 + a_3d_0
\end{bmatrix},
\]
which is equivalent to 
\begin{align}
\det(M) &= a_2^2 d_1^2 + a_1^2 d_2^2 + \beta a_3^2 d_1^2 + \beta a_1^2 d_3^2  + \gamma a_3^2 d_2^2  + \gamma a_2^2 d_3^2  \\
&= a_1^2(d_2^2 + \beta d_3^2) + a_2^2(d_1^2 + \gamma d_3^2) + a_3^2(\beta d_1^2 + \gamma d_2^2)
\end{align}

Assuming the quadratic form is $\llangle 1, \gamma \rrangle$, the following conditions are due to the norm,
\begin{align*}
a_0^2 + a_1^2 + \gamma a_2^2 +\gamma a_3^2 &= 1\\
d_0^2 + d_1^2 + \gamma d_2^2 +\gamma d_3^2 &= \gamma\\
b_1^2 + \gamma b_2^2 + \gamma b_3^2 &= 0.
\end{align*}
and by equating terms containing the non-square $\gamma$, we have the following simple conditions
\[  a_1 = a_0 + 1, \ a_3 = a_2, \ d_1 = d_0, \ d_3 = d_2 + 1, \ b_1 = 0, \ b_3 = b_2. \]

Using these, the condition on the determinant reduces significantly to
\begin{align}
\det(M) &= a_1^2(d_2^2 + d_3^2) + a_2^2(d_1^2 + \gamma d_3^2) + a_3(d_1^2 + \gamma d_2^2)\\
&= 1 + a_0^2 + \gamma a_2^2.
\end{align}
Therefore, $g$ is a bijection only if the both of these conditions are satisfied: $a_0 \ne 1$ and $a_2 \ne 0$. Applying all the above conditions to the coefficients in the map $g$ from above, we have
\begin{align}
g(e) &= e\\
g(v) &= a_0e + (a_0 + 1)v + a_2w + a_2vw\\
g(w) &= d_0e + d_0v + d_2w + (d_2 + 1)vw\\
g(vw) &= (\gamma a_2 + d_0) e + (\gamma a_2 + d_0) v + (a_0 + d_2 + 1) w + (a_0 + d_2) vw.
\end{align}
We want to choose coefficients of $g(v)$ and $g(w)$ so that $g(b) = e + n$. Observe that
\begin{align*}
g(b) &=  g(b_2 w + b_2vw)\\
&= b_2 \big[g(w) + g(vw)\big]\\
&= b_2\big(d_0e + d_0v + d_2w + (d_2 + 1)vw\big) \\
&+ b_2\big((\gamma a_2 + d_0) e + (\gamma a_2 + d_0) v + (a_0 + d_2 + 1) w + (a_0 + d_2) vw \big)\\
&= b_2\Big( \gamma a_2(e + v) + (1 + a_0)(w + vw)  \Big).
\end{align*}
From this expression, we must choose $a_2$ so that $\gamma b_2 a_2 = 1$; in other words, $a_2 = \frac{1}{\gamma b_2}$. Note that $b_2 \ne 0$ and $\gamma \ne 0$, so this is possible to do.
%

Now suppose the bilinear form is $\llangle 1, 1\rrangle$ and $\gamma = 1 = \beta$. The conditions on $g(v)$ and $g(w)$ are $a_0^2 + a_1^2 + a_2^2 + a_3^2 = 1$ and $d_0^2 + d_1^2 + d_2^2 + d_3^2 = 1$; additionally, from $\q(b) = 0$, we have $b_1^2 + b_2^2 + b_3^2 = 0$. So it follows that
\begin{align}
a_0 + a_1 + a_2 + a_3 &= 1\\
d_0 + d_1 + d_2 + d_3 &= 1\\
b_1 + b_2 + b_3 &= 0.
\end{align}
Notice that $b_1 + b_2 = b_3$, and if $b_1 + b_2 = 0$ then $b_1 = b_2$ and $b_3 = 0$. The condition on the determinant reduces to
\begin{equation}
\det(M) = a_1^2(d_2^2 + d_3^2) + a_2^2(d_1^2 + d_3^2) + a_3^2(d_1^2 + d_2^2).
\end{equation}
Then we have
\begin{align*}
g(e) &= e\\
g(v) &= a_0e + a_1v + a_2w + a_3vw\\
g(w) &= d_0e + d_1v + d_2w + d_3vw\\
g(vw) &= (a_0d_0 + a_1d_1 + a_2d_2 + a_3d_3) e \\
& \ + (a_0d_1 + a_1d_0 + a_2d_3 + a_3d_2) v\\
& \ + (a_0d_2 + a_2d_0 + a_1d_3 + a_3d_1) w\\
& \ + (a_0d_3 + a_3d_0 + a_1d_2 + a_2d_1) vw
\end{align*}
We must have $g(b) = e + n$ where $n \in \tilde{B}$. This implies that
\begin{equation}\label{ecoeff}
1 = b_1a_0 + b_2d_0 + b_3\big( a_0d_0 + a_1d_1 + a_2d_2 + a_3d_3\big).
\end{equation}
This condition depends on which field we are using; however, there are always ways to choose the coefficients of $g(v)$ and $g(w)$ in such a way that \ref{ecoeff} is satisfied.
\end{proof}


\begin{prop}
If $B$ is a division algebra, then no two elements in $B$ have the same norm.
\end{prop}
\begin{proof}
Suppose $a,b \in B$ such that $\q(a) = \q(b)$. Then
\[ \q(a+b) = \q(a) + \q(b) + \langle a,b \rangle = \q(a) + \q(b) = 0. \]
\end{proof}

\begin{cor}\label{gidcor}
If $B$ is a division algebra then $\Aut(B) = \{\id\}$.
\end{cor}

If $\mathcal{I}_t$ is $\Aut(C)$-conjugate to $\mathcal{I}_s$ then we have shown that $g(b) = c+cn$ by \ref{typeIIiso}. With this in mind we can prove the following.

\begin{prop}
Let $B$ be a division algebra and let $g \in \Aut(C)$ such that $g(B) = B$. Then $g(u) = u + m$ for $m \in \widehat{B}$.
\end{prop}
\begin{proof}
Recall that $\Aut(C) \subset \Orth(\q)$ so $\langle g(a),g(u) \rangle = \langle a,u \rangle$ for all $a \in B$.  Now since $g(B) = B$ we have by \ref{gidcor} that $g|_B = \id$.  So
\[ \langle g(a), g(u) \rangle = \langle a, u + m + nu \rangle. \]
Notice that $\langle a,m \rangle = 0$, so we have
\[ \langle a, u \rangle = \langle a, u \rangle + \langle a, mu \rangle. \]
Which gives us the relation
\[ \langle a, nu \rangle = 0, \]
for all $a \in B$.  In particular it must be true for $v \in B$, where $v^2 = \beta e$.  Computing the bilinear form of $v$ with $nu$ where $n \in \widetilde{B}$, we have
\[ \langle v, nu \rangle = \langle v, n_1 vu \rangle = \beta n_1 \langle e,u \rangle = \beta n_1 = 0. \]
Since $\beta \neq 0$, it follows that $n_1 = 0$.  A similar argument shows that $n_2 = 0$ and $n_3 = 0$. Hence $n = 0$.  This puts $m \in \widehat{B}$.
\end{proof}

This reduces the equation $b = c+cn$ to $b=c$.  So we have the following corollary. 

\begin{cor}\label{autCconjugate}
Let $B$ be a four dimensional totally singular division subalgebra of $C$ and let $t$ and $s$ be involutions of type II such that $t(u) = u+b$ and $s(u) = u+c$. Then $\mathcal{I}_t$ is $\Aut(C)$-conjugate to $\mathcal{I}_s$ if and only if $b=c$.
\end{cor}

So for each element of $\widehat{B}$ the conjugacy class consists of a single element. It is important to note that there are elements of $\widehat{B}$ that are not $0$ or $e$.

\begin{ex} \normalfont
Let $k= \mathbb{F}_2(\chi_1,\chi_2)$, and choose $\q \sim \llangle \chi_1, \chi_2 \rrangle$. Then the element
\[ \dfrac{\chi_1^2}{\chi_1^2+\chi_2} e + \dfrac{\chi_1}{\chi_1^2+\chi_2} w \in \widehat{B}. \]
To see this we take the norm
\begin{align*}
\q \left( \dfrac{\chi_1^2}{\chi_1^2+\chi_2} e + \dfrac{\chi_1}{\chi_1^2+\chi_2} w \right) &= \left(\dfrac{\chi_1^2}{\chi_1^2+\chi_2} \right)^2 + \chi_2\left(\dfrac{\chi_1}{\chi_1^2+\chi_2}\right)^2 \\
&= \dfrac{\chi_1^4}{\chi_1^4+\chi_2^2}  + \chi_2\dfrac{\chi_1^2}{\chi_1^4+\chi_2^2} \\
&= \dfrac{\chi_1^4 + \chi_1^2\chi_2}{\chi_1^4+\chi_2^2} \\
&= \dfrac{\chi_1^2(\chi_1^2+\chi_2)}{(\chi_1^2+\chi_2)(\chi_1^2+\chi_2)} \\
&= \dfrac{\chi_1^2}{\chi_1^2+\chi_2}.
\end{align*}
\end{ex}

\begin{prop}\label{divBfixBhat}
Let $B$ be a division algebra and $t$ an involution of type II such that $t(u) = u + b$. Then $\Aut(C)^{\mathcal{I}_t} \cong \widehat{B}$.
\end{prop}
\begin{proof}
Recall that, by \ref{gidcor}, $\Aut(B) = \{\id\}$.  If $g \in \Aut(C)^{\mathcal{I}_t}$, we have shown that $g(B) = B$, so $g|_B \in \Aut(B) = \{\id\}$.  This means that $g(u)$ completely determines $g$.  And we have shown that, when $B$ is a division algebra, the admissible images of $u$ under $g \in \Aut(C)^{\mathcal{I}_t}$  consist of $g(u) = u + m$ where $m \in \widehat{B}$.  So $g\in \Aut(C)^B \cong \widehat{B}$.
\end{proof}

\begin{thm}\label{classesBsummary}
Let $B$ be a totally singular four dimensional subalgebra of $C$ with a quadratic form $\q$.
\begin{enumerate}[$(1)$]
\item If $\q \sim \llangle 1,1 \rrangle$ or $\q \sim \llangle 1,\gamma \rrangle$ $\gamma \not\in k^2$, then there is one $\Aut(C)$-conjugacy class of involutions of type II fixing $B$ and one isomorphism class of $\Aut(C)^{\mathcal{I}_t} \cong \Aut(B) \ltimes \widehat{B}$ \\
\item If $\q \sim \llangle \beta,\gamma \rrangle$ $\gamma \not\in k^2$ and $\beta \not\in k^2(\gamma)$, then there is an $\Aut(C)$-conjugacy class of involutions of type II fixing $B$ for each element of $\widehat{B}$ and only one isomorphism class of $\Aut(C)^{\mathcal{I}_t} \cong \widehat{B}$.
\end{enumerate}
\end{thm}
\begin{proof}
For the cases when $\beta = 1$ from \ref{conjtoIDqbne0} and \ref{conjtoIDqb0} we know that there is only one $\Aut(C)$-conjugacy class of involutions of type II, and so by \ref{isofixpoint} there is only one $\Aut(C)$-conjugacy class of fixed point groups in $\Aut(C)$.  When $\q$ defines a division algebra $B$,  by \ref{autCconjugate} and \ref{divBfixBhat}, we have many $\Aut(C)$-conjugacy classes of involutions all having the same centralizer.
\end{proof}

This leads us into a discussion of the related Galois cohomology following the notation and conventions set by Serre in \cite{se97}. In general, the Galois cohomology groups $H^0(k,\Aut(C_K)) \cong \Aut(C_k)$ and $H^1(k,\Aut(C_K))$ give us the $K/k$-forms of $\Aut(C_K)$ over $k$.  We are interested in the cohomology group that corresponds to a class of involutions on $\Aut(C)$.  In particular we are interested in the $K/k$-forms of $\Aut(C_K,D_K)$ and $\Aut(C_K,B_K)$ the subgroups that fix our four dimensional subalgebras of $C$.   Let $K$ be the algebraic closure of $k$ then and recall that when $\ch(k) \neq 2$ all four dimensional subalgebras are quaternion subalgebras and we have 
\[ H^1(k,\Aut(C_K,D_K)) \cong H^1(k,D_K) \cong H^1(k, \Aut(C_K)^{\mathcal{I}_t}). \]
When $\ch(k) = 2$
\[ H^1(k,B_K) \cong H^1(k,\Aut(C_K,B_K)) \cong H^1(k, \Aut(C_K)^{\mathcal{I}_t}), \]
when $B$ is not a division algebra and $t$ is an involution of type II that fixes $B$ elementwise.  This gives us the following bijection.
\begin{prop}
There is a bijection between the set of unique biquadratic field extensions
\[
k^2 \subset k^2(\beta,\gamma) \subset k
\]
and elements of $H^1(k,\Aut(C_K)^{\mathcal{I}_t})$.
\end{prop}
\begin{proof}
This follows from \ref{classesBsummary} and \ref{hofflag}.
\end{proof}
When $t$ is of type I we have that the $K/k$-forms associated with $t$ are trivial since $t$ always fixes a split quaternion subalgebra of which there is only one isomorphism class.

\bibliographystyle{plain}

\end{document}